\documentclass{article}
\usepackage{amsmath, amssymb, amsthm}
\usepackage[dvipdfmx]{graphicx}

\usepackage{titlesec}
\titleformat{\section}{\centering\normalfont\normalfont\bfseries}{\thesection.}{2mm}{}

\usepackage{indentfirst}

\numberwithin{equation}{section}

\usepackage{hyperref}
\usepackage[noabbrev,capitalise,nameinlink]{cleveref}
\hypersetup{
    colorlinks=true,
    citecolor=blue,
    linkcolor=blue,
    urlcolor=blue,
}

\usepackage[
    backend=biber,
    style=alphabetic,
    sorting=anyt,
    doi=false,
    uniquelist=false,
    abbreviate=true,
    backref=true,
    maxnames=99,
    maxalphanames=4,
    minalphanames=3
]{biblatex}

\renewbibmacro{in:}{}

\DefineBibliographyStrings{english}{
  backrefpage  = {\hspace{-0.15cm}},
  backrefpages = {\hspace{-0.15cm}},
}

\theoremstyle{definition}
\newtheorem{definition}{Definition}

\newtheorem{thm}{Theorem}
\newtheorem{prop}{Proposition}
\newtheorem{lem}{Lemma}
\newtheorem{rem}{Remark}
\newtheorem{eg}{Example}

\newcommand{\C}{\mathbb C}
\newcommand{\Z}{\mathbb Z}
\newcommand{\T}{\mathbb T}
\newcommand{\I}{\mathcal I}

\def\SL{\mathrm{SL}}
\def\Tr{\mathrm{Tr}}

\title{Integrality of genus-$g$ indices with \\ 
adjoint Reidemeister torsions of twist knots}
\author{Ryoto Tange, Yuji Terashima, and Yoshikazu Yamaguchi}
\date{}

\begin{document}

\maketitle

\begin{abstract}
We consider the sum of the adjoint Reidemeister torsions and prove the integrality for twist knots and the meridian. We also give some concrete examples of the generating functions for these sums. 
\end{abstract}

\section{Introduction}
In the rich and fruitful interplay between the physics and mathematics of quantum fields, quantum field theory continues to provide new perspectives and conjectures that often appear unexpected from a mathematical point of view. Recently, Gang, Kim, and Yoon, based on the 3d–3d correspondence \cite{CCV,DG,DGG,TY1,TY2}, proposed that the index $I_g$ of a three-dimensional gauge theory on the product of a genus-$g$ surface and $S^1$, labeled by a $3$-manifold $M$, can be identified with a sum of Reidemeister torsion associated with $M$. From a physical perspective, the proposal is also intriguing in that it expresses the index—an inherently quantum and highly nontrivial observable—in terms of Reidemeister torsion, a classical topological invariant of $3$-manifolds. See also \cite{BGP,GK,GKP}. Based on this proposal, they formulated a mathematical conjecture asserting that this sum of Reidemeister torsions should be integral, as a consequence of the integrality of the index. 

In this paper, we prove their integrality conjecture in the case where $M$ is the complement of any twist knot. Furthermore, we provide an explicit formula expressing the generating function of the indices $I_g$ for all genera $g$ as a ratio of polynomials, and compute concrete examples of this generating function for specific knots such as the figure-eight knot. According to the proposal of Gang–Kim–Yoon, this also yields, from a physical viewpoint, a new method for obtaining infinitely many indices simultaneously. It would be interesting to connect the recent relationships between gauge theory and quantum knot invariants, such as the colored Jones polynomial through the gauge–quiver correspondence \cite{EKL, KRSS1, KRSS2} and the gauge–knot correspondence \cite{MTT}, with the relationship between adjoint Reidemeister torsion and gauge theory studied in this paper, via refinements of the volume conjecture.

This paper is organized as follows:
In Section 1, we recall the character varieties and the adjoint Reidemeister torsion. 
In Section 2, we prove the integrality on sums of adjoint Reidemeister torsions for twist knots and the meridian, and give some concrete examples of the generating functions for these sums.

\subsection*{Acknowledgements}
  We would like to express our gratitude to H.~Fuji, D.~Gang, M.~Manabe, and S.~Terashima for helpful discussions.
  The authors are supported by JSPS KAKENHI Grant Number 22H01117, 24K06720 and 25K06969.

\section{Preliminaries}

\subsection{Character varieties}

The twist knot $J(-2, -n)$ is illustrated as in Figure~\ref{fig:twistknot}.
Here $|n|$ crossings means $|n|$ right-handed half twists in the box (left-handed if $n$ is negative).

\begin{figure}[ht]
  \centering
  \includegraphics[scale=0.45]{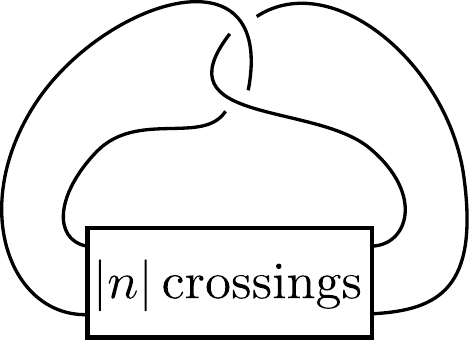}
  \label{fig:twistknot}
  \caption{the diagram of $J(-2, -n)$}
\end{figure}
We consider the topological invariants of twist knot exteriors which are determined by the knot groups. Let $K$ be a twist knot $J(-2, -2l)$ in $S^3$. 
The diagram below illustrates the case of $l=2$, 
which corresponds to the knot $5_2$: 

\begin{figure}[ht]
  \centering
  \includegraphics[scale=0.5]{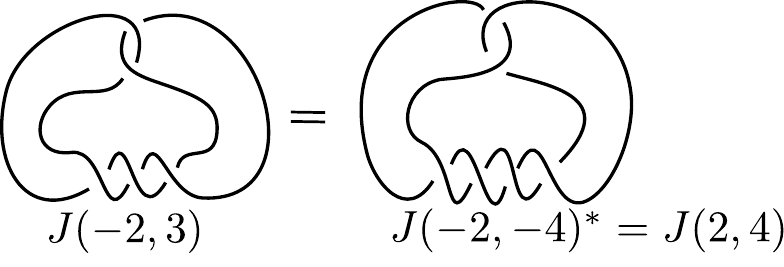}
  \caption{the diagrams of $J(-2, 3)$ and the mirror image of $J(-2, -4)$}
\end{figure}

\begin{rem} {\rm
Knot groups are isomorphic under mirror image. 
Since the twist knot $J(-2, -2l+1)$ is equivalent to the mirror image of $J(-2, -2l)$, that is $J(2, 2l)$, 
we focus on the twist knots $J(-2, -2l)$.
}
\end{rem}

The twist knot $K$ is hyperbolic except for $l=0$ and $1$. In the case of $l=0$ and $1$, $K$ is unknot and the trefoil knot respectively.
We denote by $M_K$ the knot exterior $S^3 \setminus \hbox{int}N(K)$ where $N(K)$ is a tubular neighbourhood of $K$.
The knot group $\pi_1(M_K)$ of $K$ admits the presentation
\[ \pi_1(M_K) = \langle g_1, g_2 \mid  w^{l} g_1 = g_2 w^{l} \rangle, \]
where $w := g_2 g_1^{-1} g_2^{-1} g_1$.
Note that both of $g_1$ and $g_2$ are meridians of $K$. 
For later use, we recall the defining polynomial of the $\SL_2(\C)$-character variety for twist knots.
We restrict our attention to the set of characters for irreducible representations $\rho:\pi_1(M_K) \to \SL_2(\C)$ which is regarded as the set of conjugacy classes of irreducible representations. These sets form algebraic varieties.
The defining polynomial $F(x,z) \in \mathbb{Z}[x,z]$ of the character variety can be expressed in terms of the Chebyshev polynomials as follows (\cite{Yoon2022TwistKnots}):
\[ F(x,z) = (-z+x^2-2) \cdot S_{l} \cdot (S_{l}-S_{l-1}) + 1, \]
where $x := \mathrm{Tr} \rho (g_1)$, $z := \mathrm{Tr} \rho (w)$ for an irreducible representation $\rho \colon \pi_1(M_K) \to \SL_2(\C)$, and
$S_{l} = S_{l}(z)$ denotes the Chebyshev polynomials defined by the recurrence relation 
$S_{l+1}(z) = z \cdot S_{l}(z) - S_{l-1}(z)$ for all $l \in \mathbb{Z}$ with the initial conditions $S_0(z) = 0$ and $S_1(z) = 1$. 
The leading coefficient of $F(x,z)$ with respect to the variable $z$ is always $\pm 1$. 
In particular, $F(x,z)$ is monic in $z$, and we write 
\[ F(x,z) = \pm \{ z^{m} - f_1(x) \cdot z^{m-1} + f_2(x) \cdot z^{m-2} - \cdots + (-1)^{m} \cdot f_{m}(x) \}, \] 
where $m$ is the degree of $F(x,z)$ with respect to $z$.

For example, in the case of $l=-1$, 
corresponding to the figure-eight knot $4_1$, 
we have 
\[ F(x,z) = z^2 - (x^2 - 1)z + (x^2 - 1). \] 
In the case of $l=2$, 
corresponding to the knot $5_2$, 
we have 
\[ F(x,z) = -z^3 + (x^2 - 1)z^2 - (x^2 - 2)z + 1. \]

\subsection{Adjoint Reidemeister torsions}

Next, we recall the adjoint Reidemeister torsion. 
Let $X^{\mathrm{irr}}(M_K)$ denote the character variety of irreducible representations $\pi_1(M_K) \to {\rm SL}_2(\C)$. 
For simplicity, we assume that each irreducible component of $X^{\mathrm{irr}}(M_K)$ is one-dimensional. 
In \cite{Porti1997torsion}, Joan Porti introduced the \emph{adjoint Reidemeister torsion} $\T_\gamma$ as a function defined on a Zariski open subset of $X^{\mathrm{irr}}(M_K)$, depending on the choice of a boundary curve $\gamma$. Here, a boundary curve refers to a simple closed curve in $\partial M_{K}$ representing a non-trivial class in $H_1(\partial M_K; \Z)$. 

For an irreducible representation $\rho \colon \pi_1(M_K) \to {\rm SL}_2(\C)$ with character $\chi_\rho$, the value $\T_{\gamma}(\chi_\rho)$ is given by the Reidemeister torsion associated with the adjoint action of $\rho$. 
The definition requires a choice of $\gamma$ in order to fix a basis of the corresponding twisted (co-)homology.

Let $\mu$ be the meridian of $K$. 
By \cite[Remark 1]{Yoon2022TwistKnots}, the torsion $\T_{\mu}(\chi_\rho)$ can be expressed in terms of the derivative of 
the defining polynomial $F = F(x,z)$ of the character variety 
as follows:
\[ \T_{\mu}(\chi_\rho) = \frac{1}{2} \cdot \frac{\partial F}{\partial z}. \]

\section{Integrality for twist knots and the meridian}

We prepare some notation to observe the integrality for twist knots.
The adjoint Reidemeister torsion $\T_{\mu} = (1/2) \cdot (\partial F / \partial z)$ is defined as a non-zero complex number. We restrict the domain of $\partial F / \partial z$ under the constraint $\partial F / \partial z \not =0$.
In other words, since $\partial F / \partial z$ as a function in $x$ on $F(x, z)=0$,
we restrict the domain of $x$ such that $F(x, z)=0$ has distinct roots in $z$ for each $x$.
\begin{definition}
  Let $\{ a_1, \dots, a_{m} \}$ 
  be the distinct roots $z=a_i$ of $F(x, z)=0$, 
  which correspond to the set $\{ \chi_{\rho_1}, \dots, \chi_{\rho_{m}} \}$ of generic irreducible characters. 
  Define 
  \[ b_i := d_{\mu} \cdot \T_{\mu}(\chi_\rho)|_{z=a_i} = \left. \frac{\partial F}{\partial z} \right|_{z=a_i}, \]
  where $d_\mu = 2$ is chosen to satisfy the condition of the Gang--Kim--Yoon Conjecture for the meridian case \cite[Conjecture 3.1]{GangKimYoon2021M5branes}. 
\end{definition}

The integrality for twist knots and the meridian $\mu$ can be proved by using the following Lemma:
\begin{lem} \label{lem:elemPolyb}
The elementary symmetric polynomials $e_j(b_1, \dots, b_{m})$ 
are polynomials in $x$ with integer coefficients for any $j$, 
i.e., $e_j(b_1, \dots, b_{m}) \in \Z[x]$. 
\end{lem}

\begin{proof}
Since $F(x, z)$ is monic in $z$, 
the coefficient $f_j(x) \in \Z[x]$ of $F(x,z)$ 
can be regarded as the elementary symmetric polynomials $e_j(a_1, \dots, a_{m}) \in \Z[a_1, \dots, a_{m}]$ for any $j$. 
Hence, $e_j(b_1, \dots, b_{m})$ is a symmetric polynomial in $\Z[a_1, \dots, a_{m}]$, and 
by the fundamental theorem of symmetric polynomials, 
$e_j(b_1, \dots, b_{m})$ can be considered as an element of the ring generated over $\Z$ 
by elementary symmetric polynomials $\{e_j(a_1, \dots, a_{m})\}$. 
This implies that $e_j(b_1, \dots, b_{m})$ belongs to $\Z[x]$. 
\end{proof}

Next, for $g \in \Z_{\ge 0}$, we consider 
the index $\I_{g}$ defined as follows:
\[ \I_{g} := \sum_{i=1}^m b_i^{g-1}. \]

The main result of this paper is the following theorem:

\begin{thm} \label{thm:main}
Let $K$ be a twist knot. 
Then for $g \in \Z_{\ge 2}$, the index 
$\I_{g}$ is a polynomial in $x$ with integer coefficients, i.e., $\I_{g} \in \Z[x]$
and the Gang--Kim--Yoon Conjecture holds for the meridian case. 
\end{thm}

\begin{proof}
Consider the generating function $Z(x,t)$ for $\I_{g}$, that is, 
\[
Z(x,t):= \sum_{g=2}^{\infty} \I_{g} \cdot (-t)^{g-2}.
\]
By using the generating function $E(x,t)$ of 
the elementary symmetric polynomial $e_j(b_1, \dots, b_{m})$ 
given by 
\[
E(x,t):= 1 + \sum_{j=1}^{m} e_j(b_1, \dots, b_{m}) t^j \in \Z[x,t],
\]
we have
\[
\begin{aligned}
Z(x,t)
&=\sum_{i=1}^{m} \sum_{g=2}^{\infty} b_i^{g-1} \cdot (-t)^{g-2}\\
&=\sum_{i=1}^{m} \frac{b_i}{1+b_i t}\\
&=\frac{\partial}{\partial t}\log \left( \prod_{i=1}^{m}\left(1+b_i t\right) \right)\\
&=\frac{\partial}{\partial t}\log E(x,t)\\
&=\frac{\frac{\partial}{\partial t}E(x,t)}{E(x,t)}.
\end{aligned}
\]
due to computations in the theory of symmetric functions. See \cite[Section 1.2]{Macdonald1995}. Since $E(x,t)$ is a polynomial in $t$ with coefficients in $\mathbb{Z}[x]$ starting with $1$, $1/E(x,t)$ can be expanded as a formal power series in $t$ with coefficients in $\mathbb{Z}[x]$. Therefore, 
\[
Z(x,t) = \sum_{g=2}^{\infty} \I_{g} \cdot (-t)^{g-2}.
\]
is a formal power series in $t$ with coefficients in $\mathbb{Z}[x]$, and the index $I_{g}$ is an element in $\mathbb{Z}[x]$.
\end{proof}

The proposition below, which was shown in the proof above, is a crucial result that controls the infinite indices $I_g$ for all $g \in \Z_{\ge 2}$ using a finite amount of information; therefore, we restate it here.
\begin{prop}
We have 
\[
Z(x,t)=\frac{\frac{\partial}{\partial t}E(x,t)}{E(x,t)}
\]
where $E(x,t)$ is the generating function of 
the elementary symmetric polynomial $e_j(b_1, \dots, b_{m})$ 
given by 
\[
E(x,t)= 1 + \sum_{j=1}^{m} e_j(b_1, \dots, b_{m}) t^j \in \Z[x,t]
\]
\end{prop}

For $g=1$, the sum $\mathcal{I}_1$ coincides with the number $m$ of roots of $F(x, z) = 0$ with respect to $z$.
The properties of symmetric polynomials yield the following proposition for $g=0$, which provides an alternative proof of a result by Yoon \cite{Yoon2022TwistKnots}. 

\begin{prop} \label{prop:g=0}
Let $K$ be a hyperbolic twist knot 
$J(-2,-2l)$, namely $l \neq 0, 1$. 
Then
we have 
\[ \I_0 = 0. \]
\end{prop}

\begin{proof}
Suppose that $F(x,z)$ can be factored as $F(x,z) = \pm \prod_{i=1}^{m} (z-a_i(x))$, where $m \geq 2$, and 
$a_1(x), \dots, a_{m}(x)$ are distinct roots. 
Since 
\[ \frac{\partial F}{\partial z}(x, a_j(x)) = \pm \prod_{1 \le i \le m, \ i \neq j} (a_j(x) - a_i(x)), \]
consider the sum $\I_0(x)$ expressed as
\[ \I_0(x) := \pm \sum_{j=1}^{m} \frac{1}{\prod_{i \neq j} (a_j(x) - a_i(x))}. \]
Let $V(x)$ be the Vandermonde polynomial of the $m$ roots $a_1(x), \dots, a_{m}(x)$ given by
\[ V(x) := \prod_{1 \le i < k \le m} (a_k(x) - a_i(x)) \in \mathbb{C}[a_1(x), \dots, a_{m}(x)]. \]
Then we have 
\[
\begin{aligned}
\frac{V(x)}{\prod_{i \neq j} (a_j(x) - a_i(x))} 
&= \frac{\prod_{i < j} (a_j(x) - a_i(x)) \cdot \prod_{k > j} (a_k(x) - a_j(x)) \cdot \widehat{V_j}(x)}{\prod_{i < j} (a_j(x) - a_i(x)) \cdot (-1)^{m-j} \prod_{k > j} (a_k(x) - a_j(x))} \\
&= (-1)^{m-j} \widehat{V_j}(x).
\end{aligned}
\]
where $\widehat{V_j}(x)$ denotes the Vandermonde polynomial of the $m-1$ roots omitting $a_j(x)$. 
Hence, we have 
\[
\begin{aligned}
\I_0(x) \cdot V(x) 
&= \pm \sum_{j=1}^{m} (-1)^{m-j} \widehat{V_j}(x) \\
&= \pm \det \begin{pmatrix}
1 & a_1(x) & a_1(x)^2 & \cdots & a_1(x)^{m-2} & 1 \\
1 & a_2(x) & a_2(x)^2 & \cdots & a_2(x)^{m-2} & 1 \\
\vdots & \vdots & \vdots & \ddots & \vdots & \vdots \\
1 & a_{m}(x) & a_{m}(x)^2 & \cdots & a_{m}(x)^{m-2} & 1
\end{pmatrix} \\ 
&= 0.
\end{aligned}
\]
Hence, we have $\I_0(x) = 0$ for any generic $x \in \C$ where $V(x) \neq 0$. 
Since the condition $V(x) = 0$ holds only at finitely many points in $\C$, 
$\I_0(x)$ vanishes at infinitely many points in $\C$. 
Therefore,  
we have $\I_0(x) \equiv 0$.
\end{proof}

\begin{rem} {\rm
We note that \cref{prop:g=0} does not hold for the trefoil knot $J(-2,-2)$. 
Indeed, the degree of $F(x,z)$ with respect to $z$ is $1$, and 
the torsion satisfies $\T_{\mu}(\chi_\rho) = \frac{1}{2}$ for the unique generic irreducible character $\chi_\rho$, 
which yields $\I_0 = 1$ and 
\[Z(x,t) = \frac{1}{1+t}. \]  
For the computation of $I_0$ in the case of general torus knots, see \cite{TrY}, and for the integrality of $I_g$ for arbitrary $g$ in the case of general torus knots, see \cite{TY2026GKY,MT}.  
}
\end{rem}

\begin{eg} {\rm
Let $K = J(-2, 2)$ corresponding to 
the figure-eight knot $4_1$. 
\begin{figure}[ht]
  \centering
  \includegraphics[scale=0.5]{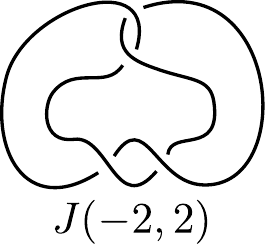}
  \caption{the diagram of $J(-2, 2)=4_1$}
\end{figure}

\noindent
Then we have 
\[ F(x,z) = z^2 - (x^2 - 1)z + (x^2 - 1) = z^2 - f_1(x) \cdot z + f_2(x) = (z -a_1) (z -a_2). \] 
Note that $f_j(x) = e_j(a_1,a_2) \in \Z[x]$ for $j = 1,2$. 
Since $\frac{\partial F}{\partial z} = 2z - f_1(x)$, 
we have 
\begin{align*}
e_1(b_1,b_2)
&= b_1 + b_2 \\
&= 2(a_1 + a_2)
   - 2 f_1(x) \\
&= 2 e_1(a_1,a_2) - 2 e_1(a_1,a_2) \\ 
&= 0, \\
e_2(b_1,b_2)
&= b_1 b_2 \\
&= 4a_1 a_2 - 2  f_1(x) \cdot (a_1 + a_2) + f_1(x)^2 \\
&= -\left(x^2-1\right)\left(x^2-5\right). 
\end{align*}
Actually $b_i$ ($i=1, 2$) turns into $\pm\sqrt{(x^2-1)(x^2-5)}$.
We also exclude the case where the set $\Tr_{\mu}^{-1}(x)$ contains reducible characters, namely when $x^2 - 5 = 0$. 
We set the domain of $x$ as $x \ne \pm 1, \pm \sqrt{5}$.

Then, we have 
\[
\begin{aligned}
E(x,t)&=1-\left(x^2-1\right)\left(x^2-5\right)t^2, \\ 
Z(x,t)
&=\frac{-2\left(x^2-1\right)\left(x^2-5\right)t}{1-\left(x^2-1\right)\left(x^2-5\right)t^2} \\ 
&=-2\left(x^2-1\right)\left(x^2-5\right)t \\
&\qquad 
\cdot \{1 + \left(x^2-1\right)\left(x^2-5\right)t^2 + 
\{ \left(x^2-1\right)\left(x^2-5\right) \}^2 t^4 + \cdots \}, 
\end{aligned}
\]
and so
\[
\I_{2h}=0, \qquad 
\I_{2h+1}= 2\left\{(x^2-1)(x^2-5)\right\}^h \qquad
(h \in \Z_{\ge 0}).
\]
}
\end{eg} 

\begin{eg} {\rm
Let $K = J(-2, -4)$ corresponding to the knot $5_2$. 
\begin{figure}[ht]
  \centering
  \includegraphics[scale=0.5]{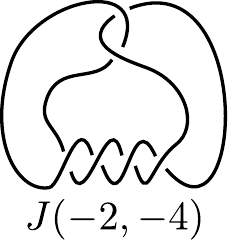}
  \caption{the diagram of $J(-2, -4)$}
\end{figure}

\noindent
Then 
we have 
\begin{align*}
  F(x,z)
  &= -z^3 + (x^2 - 1)z^2 - (x^2 - 2)z + 1 \\
  &= -z^3 + f_1(x) \cdot z^2 - f_2(x) \cdot z + f_3(x) \\
  &= -(z -a_1) (z -a_2) (z -a_3).
\end{align*}
We exclude zeros of $\partial F / \partial z$ on $F(x,z)=0$ from the domain of $x$ according to their resultant
which determines the common roots of two polynomials.
The resultant $\mathrm{Res}(F, F_z)$ of $F$ and $F_z = \partial F / \partial z$ turns out to be
\[\mathrm{Res}(F, F_z) = (x^4 - 2x^3 - 5x^2 + 14x - 7) (x^4 + 2x^3 - 5x^2 - 14x - 7).\]
We also exclude the case where the set $\Tr_{\mu}^{-1}(x)$ contains reducible characters, namely when $2x^2 - 7 = 0$. 
The domain of $x$ is defined as
\[ (2x^2 - 7)(x^4 - 2x^3 - 5x^2 + 14x - 7) (x^4 + 2x^3 - 5x^2 - 14x - 7) \ne 0.\]
Now we turn to compute $E(x, t)$ and $Z(x, t)$. 
Note that $f_j(x) = e_j(a_1,a_2,a_3) \in \Z[x]$ for $j = 1,2,3$. 
Since $\frac{\partial F}{\partial z} = - 3z^2 + 2 f_1(x) \cdot z - f_2(x)$, 
we have 
\begin{align*}
e_1(b_1,b_2,b_3)
&= b_1 + b_2 + b_3 \\
&= -3(a_1^2 + a_2^2 + a_3^2)
   + 2 f_1(x) \cdot (a_1 + a_2 + a_3)
   - 3 f_2(x) \\
&= -3\{e_1(a_1,a_2,a_3)^2 - 2 e_2(a_1,a_2,a_3)\}\\
&\qquad
   + 2 e_1(a_1,a_2,a_3) \cdot e_1(a_1,a_2,a_3)
   - 3 e_2(a_1,a_2,a_3) \\
&= - (x^4-5 x^2+7). 
\end{align*}
Similarly, we have 
\begin{align*}
e_2(b_1,b_2,b_3)
&= b_1 b_2 + b_2 b_3 + b_3 b_1 
= 0, \\
e_3(b_1,b_2,b_3)
&= b_1 b_2 b_3 
= (x^4 - 2x^3 - 5x^2 + 14x - 7)(x^4 + 2x^3 - 5x^2 - 14x - 7). 
\end{align*}
Hence, we have 
\[
\begin{aligned}
  &E(x,t)\\
  &=
1-\left(x^4-5 x^2+7\right)t 
+(x^4 - 2x^3 - 5x^2 + 14x - 7)(x^4 + 2x^3 - 5x^2 - 14x - 7)t^3, \\ 
  &Z(x,t)\\
  &=
\frac{-\left(x^4-5 x^2+7\right) 
+3(x^4 - 2x^3 - 5x^2 + 14x - 7)(x^4 + 2x^3 - 5x^2 - 14x - 7)t^2}{1- \{ \left(x^4-5 x^2+7\right)t 
-(x^4 - 2x^3 - 5x^2 + 14x - 7)(x^4 + 2x^3 - 5x^2 - 14x - 7)t^3 \} } \\ 
&= \{ - \left(x^4-5 x^2+7\right) 
+3(x^4 - 2x^3 - 5x^2 + 14x - 7)(x^4 + 2x^3 - 5x^2 - 14x - 7)t^2 \} \\ 
&\quad \cdot 
[1 + \{ \left(x^4-5 x^2+7\right)t 
-(x^4 - 2x^3 - 5x^2 + 14x - 7)(x^4 + 2x^3 - 5x^2 - 14x - 7)t^3 \} \\ 
&\quad + \{ \left(x^4-5 x^2+7\right)t 
  -(x^4 - 2x^3 - 5x^2 + 14x - 7)(x^4 + 2x^3 - 5x^2 - 14x - 7)t^3 \}^2 \\
&\quad + \cdots ] \\
&= 
  -(x^4-5x^2+7)-(x^4-5x^2+7)^2 t \\
&\qquad - (x^{12}-15x^{10}+93x^8-293x^6+471x^4-357x^2+196) t^2+ \cdots,
\end{aligned}
\]
and so
\[
\begin{aligned}
  \I_{0} &= 0, \\
  \I_{1} &= 3, \\
  \I_{2} &= -(x^4-5 x^2+7), \\
  \I_{3} &= (x^4-5 x^2+7)^2,\\
  \I_{4} &=-(x^{12}-15 x^{10}+93 x^8-293 x^6+471 x^4-357 x^2+196), \quad \dots
\end{aligned}
\]
}
\end{eg}


\printbibliography

@article{BGP,
  title={Rotating black hole entropy from $M5$-branes},
  author={Benini, Francesco and Gang, Dongmin and Pando Zayas, Leopoldo A.},
  journal={J. High Energy Phys.},
  volume={057},
  number={3},
  year={2020},
  eprint={1909.11612}
}

@article{CCV,
  title={Braids, Walls, and Mirrors},
  author={Cecotti, Sergio and Cordova, Clay and Vafa, Cumrun},
  eprint={1110.2115},
  year={2011}
}

@article{DG,
  title={Chern-Simons theory and S-duality},
  author={Dimofte, Tudor and Gukov, Sergei},
  journal={J. High Energy Phys.},
  volume={109},
  number={5},
  year={2013},
  eprint={1106.4550}
}

@article{DGG,
  title={Gauge Theories Labelled by Three-Manifolds},
  author={Dimofte, Tudor and Gaiotto, Davide and Gukov, Sergei},
  journal={Commun. Math. Phys.},
  volume={325},
  number={2},
  year={2014},
  eprint={1108.4389}
}

@article{EKL,
  title={Physics and geometry of knots-quivers correspondence},
  author={Ekholm, Tobias and Kucharski, Piotr and  Longhi, Pietro},
  journal={Commun. Math. Phys.},
  volume={379},
  number={2},
  year={2020},
  eprint={1811.03110}
}

@article{GK,
  title={Large $N$ twisted partition functions in $3$d-$3$d correspondence and holography},
  author={Gang, Dongmin and Kim, Nakwoo},
  journal={Phys. Rev. D},
  volume={99},
  number={2},
  year={2019},
  eprint={1808.02797}
}

@article{GKP,
  title={Precision microstate counting for the entropy of wrapped M$5$-branes},
  author={Gang, Dongmin and Kim, Nakwoo and Pando Zayas, Leopoldo A. },
  journal={J. High Energy Phys.},
  volume={164},
  number={3},
  year={2020},
  eprint={1905.01559}
}

@article{GangKimYoon2021M5branes,
  title={Adjoint Reidemeister torsions from wrapped M5-branes},
  author={Gang, Dongmin and Kim, Seonhwa and Yoon, Seokbeom},
  journal={Adv. Theor. Math. Phys.},
  volume={25},
  number={7},
  pages={1819--1845},
  year={2021},
  doi={10.4310/ATMP.2021.v25.n7.a4},
  eprint={1911.10718}
}

@article{KRSS1,
  title={BPS states, knots and quivers},
  author={Kucharski, Piotr and Reineke, Markus and Stosic, Marko and Sulkowski, Piotr},
  journal={Phys. Rev. D},
  volume={96},
  number={12},
  year={2017},
  eprint={1707.02991}
}

@article{KRSS2,
  title={Knots-quivers correspondence},
  author={Kucharski, Piotr and Reineke, Markus and Stosic, Marko and Sulkowski, Piotr},
  journal={Adv. Theor. Math. Phys.},
  volume={23},
  year={2019},
  eprint={1707.04017}
}

@book{Macdonald1995,
  title={Symmetric functions and Hall polynomials},
  author={Macdonald, Ian Grant},
  year={1995},
  publisher={Oxford Univ. Press}
}

@article{MT,
  title={Algebraic properties of twisted Alexander polynomial and Reidemeister torsion of torus knots},
  author={Morifuji, Takayuki and Tran, Anh T.},
  journal={arXiv:2605.22308},
  year={2026}
}

@article{MTT,
  title={The colored Jones polynomials as vortex partition functions},
  author={Manabe, Masahide and Terashima, Seiji and Terashima, Yuji},
  journal={J. High Energ. Phys.},
  volume={197},
  year={2021},
  eprint={2110.05662}
}

@book{Porti1997torsion,
  title={Torsion de Reidemeister pour les vari{\'e}t{\'e}s hyperboliques},
  author={Porti, Joan},
  volume={612},
  year={1997},
  publisher={Amer. Math. Soc.}
}

@article{TY1,
  title={$SL(2, R)$ Chern-Simons, Liouville, and Gauge Theory on Duality Walls},
  author={Terashima, Yuji and Yamazaki, Masahito},
  journal={J. High Energy Phys.},
  volume={135},
  number={08},
  year={2011},
  eprint={1103.5748}
}

@article{TY2,
  title={Semiclassical Analysis of the $3$d/$3$d Relation},
  author={Terashima, Yuji and Yamazaki, Masahito},
  journal={Phys. Rev. D},
  volume={88},
  number={2},
  year={2013},
  eprint={1106.3066}
}

@article{TY2026GKY,
  title={Gang-Kim-Yoon integrality conjectures on adjoint Reidemeister torsions for torus knots},
  author={Terashima, Yuji and Yamaguchi, Yoshikazu},
  journal={arXiv:2605.19460},
  year={2026}
}

@article{TrY,
  title={Adjoint Reidemeister torsions of once-punctured torus bundles},
  author={Tran, Anh T. and Yamaguchi, Yoshikazu},
  journal={arXiv:2109.07058},
  year={2021}
}

@article{Yoon2022TwistKnots,
  title={A vanishing identity on adjoint Reidemeister torsions of twist knots},
  author={Yoon, Seokbeom},
  journal={Algebr. Geom. Topol.},
  volume={22},
  number={1},
  pages={227--249},
  year={2022},
  doi={10.2140/agt.2022.22.227},
  eprint={2002.12576}
}


\noindent
Ryoto Tange rtange.math@gmail.com \\ 
Department of Mathematics, Rikkyo University, 3-34-1 Nishi-Ikebukuro, Toshima-ku, 171-8501, Tokyo, Japan

\

\noindent
Yuji Terashima yujiterashima@tohoku.ac.jp \\ 
Graduate School of Science, Tohoku University, 6-3 Aoba, Aramaki-aza, Aoba-ku, Sendai, 980-8578, Japan

\

\noindent
Yoshikazu Yamaguchi shouji@waseda.jp \\ 
Faculty of Commerce, Waseda University, 1-6-1 Nishiwaseda,  Shinjuku-ku, Tokyo, 169-8050, Japan

\end{document}